\documentclass[12pt]{amsart}
\usepackage{amssymb,verbatim,amscd,amsmath,graphicx,enumerate}
\usepackage{amsmath}
\usepackage{graphicx}
\usepackage{caption}
\usepackage{subcaption}
\usepackage{pdfsync}
\usepackage{fullpage}
\usepackage{color}
\usepackage{marginnote}
\usepackage{tikz}

\usepackage[margin=1 in]{geometry}
\usetikzlibrary{patterns}
\usetikzlibrary{calc}
\usetikzlibrary{matrix}
\usepgflibrary{arrows}
\usetikzlibrary{arrows}
\usetikzlibrary{decorations.pathreplacing}
\usetikzlibrary{decorations.pathmorphing}
\usepackage{hyperref}
\usepackage{cleveref}

\usetikzlibrary{arrows.meta, positioning}

\numberwithin{equation}{section}
\newtheorem{theorem}{Theorem}[section]
\newtheorem{lemma}[theorem]{Lemma}
\newtheorem{proposition}[theorem]{Proposition}
\newtheorem{corollary}[theorem]{Corollary}
\newtheorem{conjecture}[theorem]{Conjecture}

\newtheorem{setting}[theorem]{Setting}
\theoremstyle{definition}
\newtheorem{definition}[theorem]{Definition}

\newtheorem{def-prop}[theorem]{Definition-Proposition}
\newtheorem{remark}[theorem]{Remark}
\newtheorem{example}[theorem]{Example}

\newtheorem*{Mysketch}{Sketch of proof}

\DeclareMathOperator{\Ann}{Ann}
\DeclareMathOperator{\Ass}{Ass}
\DeclareMathOperator{\Min}{Min}

\DeclareMathOperator{\depth}{depth}

\DeclareMathOperator{\supp}{supp}

\def\m{{\mathfrak m}}

\def\NN{\mathbb{N}}

\def\FF{\mathbb{F}}

\def\G{\mathcal{G}}

\newcommand{\bfx}[1]{\ensuremath{{\bf{x^{#1}}}}}
\newcommand{\KK}{\mathbb{K}}

\begin{document}

\title{Regular sequences of linear forms on monomial ideals}

\author{Louiza Fouli}
\address{Department of Mathematical Sciences \\
New Mexico State University\\
P.O. Box 30001 \\
Department 3MB \\
Las Cruces, NM 88003}
\email{lfouli@nmsu.edu}
\urladdr{ https://sites.google.com/view/louiza-fouli/home}

\author{T\`ai Huy H\`a}
\address{Department of Mathematics \\
Tulane University \\
6823 St. Charles Avenue \\
New Orleans, LA 70118}
\email{tha@tulane.edu}
\urladdr{http://www.math.tulane.edu/~tai/}

\author{Susan Morey}
\address{Department of Mathematics \\
Texas State University\\
601 University Drive\\
San Marcos, TX 78666}
\email{morey@txstate.edu}
\urladdr{https://faculty.txst.edu/profile/1922542}


\begin{abstract}
In this paper we establish a means of using the combinatorics associated to a general monomial ideal $I$ in a polynomial ring $R$ to find a regular sequence of linear forms on $R/I$. The sequence of linear forms provides an effective lower bound on $\depth(R/I)$. When $I$ is the edge ideal of a graph, we provide conditions under which this bound is an equality, allowing the realization of the depth via a regular sequence of linear polynomials. In addition, we explicitly describe the minimal primes of $(I, f_1, \ldots, f_q)$, when $f_1, \ldots, f_q$ are homogeneous polynomials of degree one with pairwise disjoint support and $I$ is any monomial ideal. Finally, we propose a conjecture on the form of all associated primes of the ideal $(I, f_1, \ldots, f_q)$, when $I$ is the edge ideal of a graph and $f_i$ are disjoint stars on $I$.
\end{abstract}

\maketitle

\section{Introduction}

The depth of a module is one of the most important invariants in commutative algebra (\cite{AB, Grothendieck, HH}). It is closely related to local cohomology, projective dimension, Castelnuovo--Mumford regularity, and Cohen--Macaulayness (\cite{BrunsHerzog, Eisenbud, HH, Serre}). Although depth admits several homological descriptions, its most concrete interpretation is given by the maximum length of a \emph{regular sequence} (cf. \cite{BrunsHerzog}). Consequently, the construction of explicit regular sequences has long been a central problem in commutative algebra.

For monomial ideals, many algebraic invariants admit combinatorial descriptions. In contrast, explicit regular sequences are often difficult to construct. While depth can frequently be computed indirectly through homological methods, explicit regular sequences provide additional structural information and often reflect underlying combinatorial properties of the ideal. In particular, they yield concrete realizations of depth and frequently lead to information about associated primes and local cohomology.

In \cite{FHM}, the notion of an \emph{initially regular sequence} was introduced as a combinatorial method for producing lower bounds on depth. These sequences are constructed using initial ideals and Gr\"obner basis techniques and provide a flexible mechanism for generating candidates for regular sequences. In many examples arising from graphs, hypergraphs, and monomial ideals, the initially regular sequences constructed in \cite{FHM} were observed to be genuinely regular. This naturally leads to the question of when a collection of initially regular linear forms is in fact a regular sequence.

The purpose of this paper is to investigate this question for a family of linear forms that arise naturally in the theory of initially regular sequences. Let $I$ be a monomial ideal in a polynomial ring $R$. Following \cite{FHM}, a \emph{star} on $I$ is a linear form
\[
f=x_0+x_1+\cdots+x_t
\]
whose variables satisfy a divisibility condition with respect to the minimal monomial generators of $I$. Moreover, if $x_i$ has degree at most one in the generators of $I$ for all $1\le i\le t$, then $f$ is a star \emph{of degree one}. Stars arise naturally in combinatorial constructions of initially regular sequences. For edge ideals, they are closely related to local graph-theoretic configurations such as leaves, whiskers, and dominating structures.

While a single star of degree one is regular modulo a monomial ideal \cite{FHM}, collections of stars need not form regular sequences. The principal new idea of this paper is a combinatorial condition, called \emph{separation}, which controls the interaction among stars and eliminates certain algebraic dependencies that obstruct regularity.
Our first main result shows that separation guarantees regularity.

\medskip

\noindent
\textbf{Theorem~\ref{thm: stars reg seq}.}
\emph{Let $I$ be a monomial ideal in a polynomial ring over an algebraically closed field. If $f_1,\ldots,f_q$ are pairwise separated stars of degree one on $I$, then $f_1,\ldots,f_q$ is a regular sequence on $R/I$.}

\medskip

Theorem~\ref{thm: stars reg seq} provides a broad class of explicit regular sequences arising directly from the combinatorics of a monomial ideal. In particular, it identifies a large family of initially regular sequences that are genuinely regular and therefore produces effective lower bounds for depth.

A second theme of the paper concerns the structure of the quotient
\[
R/(I,f_1,\ldots,f_q).
\]
Although $I$ is monomial, the ideal $(I,f_1,\ldots,f_q)$ is generally not. Consequently, many of the combinatorial techniques available for monomial ideals are no longer directly applicable. This naturally raises the question of how the associated primes of a monomial ideal behave after adjoining suitable linear forms.

As a first step in this direction, we determine the minimal primes of ideals of the form $(I,f_1,\ldots,f_q)$. More precisely, we show that every minimal prime of $(I,f_1,\ldots,f_q)$ is obtained from a minimal prime of $I$ by adjoining the linear forms $f_1,\ldots,f_q$, \Cref{lem: minimal prime format}. For edge ideals we obtain stronger results. Under additional hypotheses, we prove that minimal primes of the original edge ideal give rise to associated primes of the quotient by the stars, \Cref{p:binomial version}, \Cref{thm: minimal primes extend}.

These results suggest a broader phenomenon. Motivated by computational evidence and the results obtained in Section~\ref{sec.AssPrimes}, we formulate the following conjecture.

\medskip

\noindent
\textbf{Conjecture~\ref{conj.ass}.}
\emph{Let $I=I(G)$ be the edge ideal of a graph $G$ in a polynomial ring $R$ over a field. Let $f_1,\ldots,f_q$ be separated stars on $I$. Then \begin{eqnarray*}
    P\in \Ass(R/(I,f_1,\ldots,f_q))  \text{ if and only if } P=(Q,f_1,\ldots,f_q),
\end{eqnarray*}
for some  $Q\in \Ass(R/I)=\Min(R/I)$.}

\medskip

The conjecture predicts that, despite no longer being a monomial ideal, the associated prime structure of $(I,f_1,\ldots,f_q)$ continues to retain a strong memory of the original monomial ideal. The study of associated primes is closely connected to the problem of realizing depth. Indeed, if $f_1,\ldots,f_q$ is a regular sequence on $R/I$, then
\[
\depth(R/I)
=
q+\depth(R/(I,f_1,\ldots,f_q)).
\]
Consequently, the lower bound provided by Theorem~\ref{thm: stars reg seq} is sharp precisely when the quotient by the stars has depth zero. Equivalently, one seeks conditions ensuring that the homogeneous maximal ideal becomes an associated prime after quotienting by the regular sequence.

Our final main result establishes such a criterion for edge ideals.

\medskip

\noindent
\textbf{Theorem~\ref{prop: All variables covered}.}
\emph{Let $I=I(G)$ be the edge ideal of a graph in a ring $R$ with homogeneous maximal ideal $\m$.  Let $f_1,\ldots,f_q$ be disjoint stars whose supports cover all variables. If the product of the centers of the stars is not contained in $(I,f_1,\ldots,f_q)$, then $\m \in \Ass(R/(I,f_1,\ldots,f_q))$ and 
$\depth(R/I)=q$.}

\medskip

The proofs combine combinatorial and homological techniques. A key ingredient is the analysis of colon ideals associated to stars, which allows questions about regularity to be reduced to monomial computations. These colon computations are then combined with short exact sequences and associated-prime techniques to transfer information from monomial ideals to the non-monomial quotients obtained after adjoining stars.

The paper is organized as follows. In Section~\ref{sec.LinearForm}, we introduce stars and the separation condition, prove Theorem~\ref{thm: stars reg seq}, and establish several extension results showing how separated stars can be combined with other regular sequences. In Section~\ref{sec.AssPrimes}, we investigate when sequences of stars realize depth, determine the structure of minimal primes of ideals of the form $(I,f_1,\ldots,f_q)$, establish partial results concerning associated primes, and formulate conjectures describing the associated-prime structure of these quotients.

\medskip

\noindent
\textbf{Acknowledgment.} Part of this work was done when the first two authors visited the third author at Texas State University. The authors thank Texas State University for its hospitality. The second author is partially supported by a Simons Foundation grant (\# 850912).

%%%%%%%%%%%%%%%%%%%%%%%%%%%%%%%%%%

\section{Linear Sums that Form Regular Sequences} \label{sec.LinearForm}

Let $R=\KK[x_1, \ldots, x_n]$ be a polynomial ring over a field $\KK$ and assume $I$ is a monomial ideal in $R$. 
In this section, we identify conditions under which certain homogeneous degree one polynomials will form a regular sequence on $R/I$. In \cite{FHM}, it was shown that certain linear sums, referred to as ``stars'' due to their graphical representation in the square-free case, form an initially regular sequence when they are disjoint, or under other mild hypotheses. In this section, we will show that these sequences are actually regular sequences when the stars satisfy a separability condition.

We begin by setting some standard notation and terminology for monomials and polynomials that will be used throughout the paper.
Given nonnegative integers $a_1, \ldots , a_n$, we write the monomial $x_1^{a_1}\cdots x_n^{a_n}$ in vector form $x_1^{a_1}\cdots x_n^{a_n}={\bf{x}}^{\bf{a}}$ where ${\bf{a}}=(a_1, \ldots, a_n)\in \NN^{n}$, allowing $0\in \NN$. For any polynomial $f$ in $R$, there exists a unique representation $f=\sum_{j=1}^{r}\alpha_{j}{\bf{x}}^{\bf{a_{j}}}$ where for $j\in [r]$, $0\neq \alpha_j\in \KK$ and ${\bf{a}}_j\in \NN^n$ with ${\bf{a}}_j \ne {\bf{a}}_i$ when $i\ne j$. We will refer to the elements ${\bf{x}}^{\bf{a_j}}$ as the monomials of $f$ and the elements $\alpha_{j}{\bf{x}}^{{\bf{a}}_{j}}$ as the terms of $f$. The \emph{support} of $f$, $\supp(f)$, is the set of all ${\bf{x}}^{\bf{a_{j}}}$ with $\alpha_j\neq 0$ in the representation of $f$. Finally, for a positive integer $m$, we will write $[m]$ for the set $\{1, \ldots, m\}$.

Let $I$ be a monomial ideal. The unique set of minimal monomial generators of $I$ will be denoted by $\mathcal{G}(I)$.
Given a variable $x$ of $R$, set
$$d_x(I)= \max\{t \in \mathbb{N} \, : \, x^t \mid M \,\, {\mbox{\rm for some }} \ M \in \G(I)\}.$$ 
Using this notation, we define 
a generalized version of the notion of a star on an arbitrary monomial ideal.

\begin{definition}
Let $I$ be a monomial ideal in a polynomial ring $R$ over a field $\KK$ and let $t\ge 1$. A linear sum of variables $f=x_0+x_1+\ldots +x_t$ is called a {\emph{star}} on $I$ with \emph{center} $x_0$ if whenever $x_0\mid M$ for some $M\in \mathcal{G}(I)$, then there exists $i\in [t]$ such that $x_i\mid M$. Moreover, a star $f$ on $I$ is said to be \emph{of degree one} if $d_{x_i}(I)=1$ for all $i\in [t]$. 
\end{definition}

Note that for a star of degree one, the degree condition does not apply to the variable that is the center of the star.
By convention, when the ideal $I$ is understood, we say $f$ is a star and omit the phrase ``on $I$." Moreover, if $x_0$ is a \emph{free variable}, that is, a variable such that $x_0\nmid M$ for every $M\in \mathcal{G}(I)$, then we consider $x_0$ to be a star, which is by default a star of degree one. 

Generalizing the notion of a graph or hypergraph with disjoint components, we say an ideal $I$ has \emph{disjoint components} if there exist disjoint sets of variables $B_1, B_2, \ldots, B_k$ such that:
\begin{enumerate}[(a)]
\item $R=\KK[B_1, B_2,\ldots, B_k]$;
\item $I=I_1R+I_2R + \cdots + I_kR$, where $I_i\subseteq \KK[B_i]$ for $ i\in[k]$.
 \end{enumerate}
When $I$ is a monomial ideal that has disjoint components, we say that two stars $f_1$ and $f_2$ are in \emph{different components}  
if $f_1\in \KK[B_i]$ and $f_2 \in \KK[ B_j]$ for $i \ne j$. Similarly, we say that two collections of stars, $f_1, \ldots, f_t$ and $g_1, \ldots, g_s$ are in different components if $f_i,g_j$ are in different components for all $i \in [t]$ and $j\in [s]$.

When working with stars, we frequently assume the stars are disjoint. However, for general monomial ideals an additional separability condition, which we now define, will be needed to ensure they form a regular sequence. 

\begin{definition}\label{d:separated}
Let $I$ be a monomial ideal in a polynomial ring $R$ over a field $\KK$. Let $f_1=x_0+\ldots +x_t$ and $f_2=y_0+\ldots +y_s$ be two stars on $I$ with centers $x_0$ and $y_0$, respectively. We say that $f_1$ and $f_2$ are \emph{separated} if
\begin{enumerate}[(i)]
\item $\{x_0, \ldots, x_t\} \cap \{y_0, \ldots, y_s\}=\emptyset$; 
\item For all $i\in[s]$, all $j\in [t]$, and all $M \in \mathcal{G}(I)$, $x_0y_i \nmid M$ and $y_0x_j \nmid M$.
\end{enumerate}
We say that a set of stars is separated if they are pairwise separated.
\end{definition}

When $I$ is the edge ideal of a graph, two stars are separated if and only if they are disjoint, in which case
their centers are distance at least three apart. However, when $I$ is an arbitrary monomial ideal separated is not equivalent to disjoint, as can be seen in the next example.

\begin{example}
  Let  $R=\KK[x_1,\ldots,x_7]$ be a polynomial ring over a field $\KK$ and let $I=(x_1x_3x_4, x_1^2x_3x_7, x_2x_5x_7, x_2x_6)$ be an $R$-ideal.
   Then $f_1=x_1+x_3, f_2=x_2+x_5+x_6$ are both stars of degree one on $I$ and they are separated. Note $g_2=x_2+x_6+x_7$ is also a star of degree one and although $f_1$ and $g_2$ involve disjoint variable sets they are not separated since $x_1x_7$ divides $x_1^2x_3x_7$. Notice that both $f_2$ and $g_2$ have the same center, namely $x_2$. Thus regardless of the definition of distance selected for hypergraphs, distance of centers alone does not determine separation of stars.
\end{example}

Our main tools in proving that a sequence $f_1, \ldots, f_q$ is regular will involve use of the standard short exact sequence
$$0 \rightarrow R/(J:z) \rightarrow R/J \rightarrow R/(J,z) \rightarrow 0$$
for a carefully chosen ideal $J$ and appropriate monomial $z$, and an examination of the associated primes of the terms of this sequence. To that end, the next lemma provides conditions under which the colon operation will commute with the addition of polynomials to a monomial ideal. This will allow us to compute the colon operation in the monomial setting.

\begin{lemma}\label{colon with neighbor}
Let $I\subset R=\KK[x_1, \ldots, x_n]$ be a monomial ideal and let $f_1, \ldots, f_s$ be polynomials in $R$.  
If $\bfx{b}$ is a monomial in $R$ such that $\gcd(\bfx{b}, \bfx{c})=1$ for all $\bfx{c} \in \supp(f_i)$ and $i\in [s]$, then $(I, f_1, \ldots, f_s):{\bfx{b}}=((I:\bfx{b}), f_1, \ldots, f_s).$
\end{lemma}

\begin{proof}
It is clear that $((I:\bfx{b}), f_1, \ldots, f_s) \subseteq (I, f_1, \ldots, f_s):\bfx{b}$. 

To prove the reverse inclusion, let $g\in (I, f_1, \ldots, f_s):\bfx{b}$. Write $f_i=\sum \limits_{j=1}^{r_i}\alpha_{ij}{\bf x^{a_{ij}}}$
and $g=\sum_{i=1}^{\ell}\gamma_i{\bfx{c_i}}$, where $\supp(f_i)=\{ {\bf x^{a_{ij}}}\}_{j=1}^{r_i}$ and $\supp(g) = \{ \bfx{c_i}\}_{i=1}^\ell$.  Then for some $p_i,q_i \in R$, we have
\begin{equation}\label{1}
g\bfx{b}=\sum \limits_{i=1}^{\ell}\gamma_i{\bfx{c_i}}\bfx{b}=\sum \limits_{i=1}^{p}p_i\bfx{u_i}+q_1f_1+\ldots +q_{s}f_s
\end{equation}
where $\mathcal{G}(I)=\{\bfx{u_1}, \ldots, \bfx{u_p}\}$.

For each $i\in [s]$ write $q_{i}=\sum_{j=1}^{s_i}\beta_{ij}\bfx{d_{ij}}$, where $\beta_{ij}\in \KK$ and $\supp(q_i) = \{\bfx{d_{ij}}\}_{j=1}^{s_i}$. Let 
$$\mathcal{A}_i=\{j \mid \bfx{d_{ij}}=\bfx{b}\bfx{d_{ij}'} \text{ for some monomial } \bfx{d_{ij}'} \}$$ and notice that  
$$(g-\sum \limits_{i=1}^s (\sum\limits_{j \in \mathcal{A}_i}\beta_{ij}\bfx{d_{ij}'}f_i))\bfx{b} \in (I,f_1, \ldots, f_s).$$ Therefore, it suffices to show that $g-\sum \limits_{i=1}^s (\sum\limits_{j \in \mathcal{A}_i}\beta_{ij}\bfx{d_{ij}'}f_i))\in (I:\bfx{b})$. Thus by abuse of notation, we may replace $g$ by $g-\sum \limits_{i=1}^s (\sum\limits_{j \in \mathcal{A}_i}\beta_{ij}\bfx{d_{ij}'}f_i))$ and thereby assume that $\mathcal{A}_i = \emptyset$ for all $i \in [s]$. That is, in \eqref{1}, we may assume that $\bfx{b}$ does not divide $\bfx{d_{ij}}$ for any $i,j$. To complete the argument, we must show that under these conditions, $g\in (I:\bfx{b})$. 
Since $(I:\bfx{b})$ is a monomial ideal, it suffices to show that $\bfx{c_i}\in (I:\bfx{b})$ for all $i\in [\ell]$. 

Let $i\in [\ell]$. By \eqref{1} $\bfx{c_i}\bfx{b}$ must be a monomial of either $p_j\bfx{u_j}$ or $q_jf_j$
for some $j$. If $\bfx{c_i}\bfx{b}$ is a monomial of $p_j\bfx{u_j}$, then $\bfx{u_j}$ divides $\bfx{c_i}\bfx{b}$ and thus $\bfx{c_i}\in (I:\bfx{b})$. If on the other hand, if $\bfx{c_i}\bfx{b}$ is a monomial of $q_jf_j$, then   $\bfx{c_i}\bfx{b}=\bfx{d_{jt}}\bfx{a_{jw}}$ for some $t,j,w$. Since $\bfx{b}\nmid \bfx{d_{jt}}$, then $\gcd(\bfx{b}, \bfx{a_{jw}})\neq 1$, contradicting the hypotheses.
\end{proof}

We remark here that for a monomial ideal $I$, any star of degree one is regular on $R/I$ by \cite[Theorem~3.11]{FHM}. When forming regular sequences, the first step is to examine stars in different components. To do so, we apply a theorem that assumes the base field $\KK$ is algebraically closed. Thus for the remainder of the section we assume that $\KK$ is algebraically closed. In general, this is a reasonably mild assumption in our setting. If $\KK$ is not algebraically closed, then since passing to the algebraic closure is a faithfully flat extension, the map $R=\KK[x_1, \ldots, x_n] \rightarrow S=\overline{\KK}[x_1, \ldots, x_n]$ is also faithfully flat. Then a sequence of homogeneous polynomials $f_1, \ldots, f_q$ is regular over $R/I$ if and only if the sequence of images under the extension, $\overline{f_1}, \ldots, \overline{f_q}$, is regular over $R/I \otimes_R S$ (see for example, \cite[Tag 00LM]{stacks-project}).

In order to work with stars in different components of a monomial ideal $I$, we first need a basic extension property for regular elements.

\begin{lemma} \label{l:comp}
    Let $I=I_1R+I_2R$ be an ideal in a polynomial ring $R=\KK[B_1,B_2]$ over an algebraically closed field $\KK$, where $B_1 \cap B_2 = \emptyset$ and $I_i \subset \KK[B_i]$ for $i=1,2$. If $f\in \KK[B_1]$ is regular on $\KK[B_1]/I_1$, then $f$ is regular on $R/I$. 
\end{lemma}

\begin{proof}
Since $\KK$ is algebraically closed, for any $P\in \Ass(R/I)$ we can write $P=P_1R+P_2R$, with $P_i\in \Ass(\KK[B_i]/I_i)$ for $i=1,2$, by \cite[Corollary~2.8]{HNTT}. 
 Assume $f \in P_1R + P_2R$. Then $f \in \KK[B_1] \cap (P_1R + P_2R) = P_1$, since $P_2 \subseteq \KK[B_2]$. Furthermore, $f$ is regular on $\KK[B_1]/I_1$, so $f \not\in P_1$, a contradiction. 
\end{proof}

Using \Cref{l:comp}, the next two results show that over an algebraically closed field, a set of stars of degree one coming from different components of an ideal $I$ will form a regular sequence.

\begin{corollary}\label{c:different components}
Let $I$ be a monomial ideal in a polynomial ring $R$ over an algebraically closed field $\KK$. If $f_1, \ldots, f_{q-1}$ is a regular sequence on $I$  and $f_q$ is a star of degree one in a different component from $f_1, \ldots, f_{q-1}$, then $f_1, \ldots, f_q$ is a regular sequence on $R/I$.
\end{corollary}

\begin{proof}
Since $f_q$ and $f_1, \ldots , f_{q-1}$ are in different components, there exist disjoint variable sets $B_1, B_2$ with $f_1, \ldots, f_{q-1} \in \KK[B_1]$, $f_q \in \KK[B_2]$, and $I=I_1R + I_2R$ where $I_j \subset \KK[B_j]$ for $j=1,2$. Suppose that $f_1, \ldots, f_{q-1}$ is regular on $R/I$. Consider $J=(I,f_1, \ldots, f_{q-1})$. By definition, $J=J_1R+I_2R$ where $J_1=(I_1,f_1, \ldots, f_{q-1})$. Note that $f_q$ is regular on $\KK[B_2]/I_2$ by \cite[Theorem~3.11]{FHM}. By \Cref{l:comp}, $f_q$ is then regular on $R/J$ and the assertion follows. 
\end{proof}

\begin{corollary}\label{c:stars in different components}
   Let $I$ be a monomial ideal in a polynomial ring over an algebraically closed field. If $f_1, \ldots, f_q$ are stars of degree one in different components of a monomial ideal $I$, then $f_1, \ldots, f_q$ is a regular sequence on $R/I$. 
\end{corollary}

\begin{proof} The proof follows by induction, making use of \cite[Theorem~3.11]{FHM} and Corollary \ref{c:different components}.
\end{proof}

We can now generalize \Cref{c:stars in different components} to arbitrary sequences of stars of degree one not necessarily in different components of $I$. We start with a simple observation that will be used regularly in the proofs without citation.

\begin{remark} \label{r:assExt}   
Let $R'$ be a polynomial ring and $R=R'[z]$ where $z$ is an indeterminate. If $A$ is an ideal of $R'$ and $P$ is an associated prime of $R/AR$, then $P$ has the form $P'R$, where $P'$ is an associated prime of $A$ in $R'$. To see this, first note that if $P'$ is a prime ideal in $R'$, then $P'R$ is a prime ideal in $R$ by the prime extension property.
By \cite[Theorem 12]{MatOld}, the set of associated primes of $R/AR = R'/A \otimes_{R'} R$ is the union of $\Ass(R/P'R)$, where $P'$ is an associated prime of $R'/A$. However, as mentioned above, if $P'$ is in $\Ass(R'/A)$, then $P'R$ is a prime ideal in $R$, so $\Ass(R/P'R) = \{P'R\}$ consists of just one ideal $P'R$.
\end{remark}

\begin{theorem} \label{thm: stars reg seq}
Let $I$ be a monomial ideal in a polynomial ring $R$ over an algebraically closed field. Let $f_1, \ldots, f_q$ be stars of degree one on $I$ that are pairwise separated. Then $f_1, \ldots, f_q$ is a regular sequence on $R/I$. 
\end{theorem}

\begin{proof}
We proceed by induction. Note that $f_1$ is regular by \cite[Theorem~3.11]{FHM}, so if $q=1$ the statement holds. Now suppose that $q>1$ and for some $1\le s<q$ we have that $f_1, \ldots, f_s$ is a regular sequence on $R/I$. We want to show that $f_{s+1}$ is regular on $R/(I, f_1, \ldots, f_s)$. For simplicity write $f_{s+1}=z_0+\ldots+z_t$ and $J=(I, f_1, \ldots, f_s)$. We will prove 
that $f_{s+1}\not\in P$ for any $P \in \Ass(R/J)$.

If $t=0$, then $z_0$ is a free variable on $I$, that is $z_0\nmid M$ for all $M\in \G(I)$. Moreover, since $f_{s+1}=z_0$ is separated from $f_i$ for all $i\le s$, by Lemma \ref{l:comp}, we have $z_0$ is regular on $R/J$ and $z_0 \not\in P$ for any $P \in \Ass(R/J)$. Assume that  $t>0$. 

Consider the following exact sequence
\begin{eqnarray}\label{ses z_1}
0\rightarrow R/(I, f_1, \ldots, f_s):z_t\rightarrow R/J \rightarrow R/(I, f_1, \ldots, f_s,z_t)\rightarrow 0.
\end{eqnarray}
It follows that
$$\Ass(R/J) \subseteq \Ass(R/[(I,f_1, \ldots, f_s):z_t]) \cup \Ass(R/(I,f_1, \ldots, f_s,z_t)).$$ 
By  Lemma~\ref{colon with neighbor}
\begin{eqnarray*} 
A=(I, f_1, \ldots, f_s):z_t&=&((I:z_t), f_1, \ldots, f_s)=(I, N_I(z_t), f_1, \ldots, f_s),
\end{eqnarray*}
where $N_I(z_t)=\{ M \in R\mid Mz_t\in \mathcal{G}(I)\}$. 
Notice that $z_t$ is a free variable on $R/A$ and, therefore, by Remark \ref{r:assExt}, $f_{s+1}=z_0+z_1+\ldots+z_t\not\in P$ for any $P\in \Ass(R/A)$. 

Next we consider the ideal $B=(I, f_1, \ldots, f_s, z_t)$. We will show that $f_{s+1}$ is regular on $R/B$. We proceed by induction on $t$. Suppose $t=1$. That means $f_{s+1}=z_0+z_1$ and for every monomial $M\in \mathcal{G}(I)$ such that $z_0\mid M$, then $z_1\mid M$. After removing any redundant monomial generators of $B$ that are divisible by $z_1$, we may assume that 
there is no minimal monomial generator in $B$ that is divisible by $z_0$. Since $z_0$ does not appear in $f_1, \ldots, f_s$, the generators of $B$ are in $R'$, where $R=R'[z_0]$, and
$R/B \cong (R'/B)[z_0]$. Particularly, $z_0$ is regular on $R/B$ and so $z_0\not\in P$ for any $P\in \Ass(R/B)$.  Moreover, $z_1\in P$ for all $P\in \Ass(R/B)$. Hence $f_{s+1}\not \in P$ for any $P\in \Ass(R/B)$. From the short exact sequence (\ref{ses z_1}), we conclude that $f_{s+1} \not \in P$ for any $P\in \Ass(R/J)$. In particular, $f_{s+1}$ is regular on $R/J$. 

Suppose now that $t>1$. By induction we may assume that if $f_1, \ldots, f_s$ are pairwise separated stars of degree one on a monomial ideal $K$ and $g$ is a star of degree one on $K$ such that $g$ is separated from $f_1, \ldots, f_s$, then $g$ is regular on $R/(K, f_1, \ldots, f_s)$ provided that $g$ is a star of degree one in at most $t$ variables.  
Let $f_{s+1}'=z_0+\ldots+z_{t-1}$ and notice that $f_{s+1}'$ is a star of degree one on $K$, where $K=(I, z_t)$. Since $f_1, \ldots, f_{s+1}$ are pairwise separated on $I$, we have $f_1, \ldots, f_{s}, f_{s+1}'$ are pairwise separated stars on $K$. By induction, $f_{s+1}'$ is regular on $R/B$, since $B = (K,f_1, \ldots, f_s)$. This implies that $f_{s+1}' \not\in P$ for all $P\in \Ass(R/B)$. Moreover, since $z_t \in K$, we have $z_t\in P$ for any $P\in \Ass(R/B)$. It follows that $f_{s+1}=f_{s+1}'+z_t\not \in P$ for all $P\in \Ass(R/B)$.
 Therefore, by the short exact sequence (\ref{ses z_1}) as before, we deduce that $f_{s+1}$ is regular on $R/J$, as claimed. 
\end{proof}

Next we show that the regular sequences from \Cref{thm: stars reg seq} can at times be combined with other forms of linear regular sequences to form longer regular sequences.

\begin{corollary}\label{c:binom-extn} Let $I$ be a monomial ideal in a polynomial ring $R$ over an algebraically closed field. Let $f_1, \ldots, f_q$ be stars of degree one on $I$ that are pairwise separated. Let $\{g_i=a_i+b_i\}_{i=1}^{\ell}$ be a regular sequence of binomials on $R/I$ for some $\ell\ge 1$ such that for all $i\in [\ell]$, 
$\supp(g_i)\cap \supp(f_j)=\emptyset$ for all $i\in [\ell]$ and all $j\in [q]$.
Then $g_1, \ldots, g_{\ell}, f_1, \ldots, f_q$ is a regular sequence on $R/I$. 
\end{corollary}

 \begin{proof}
 By \cite{FHM}, if $g_1 = a_1+b_1, \ldots, g_\ell = a_\ell+b_\ell$ is a regular sequence of binomials, then 
$$R/(I,g_1,\ldots,g_\ell) \cong R'/I'$$
where $I'$ is formed from $I$ by, for $1\le i \le \ell$, replacing $b_i$ by $a_i$ in each generator of $I$, and $R' \cong R/(b_1, \ldots, b_t)$. Since for all $i\in [\ell]$, $\{a_i, b_i\}$ is disjoint from any of the variables in the stars $f_1, \ldots, f_q$, passing to $I'$ does not alter the degree of any variable in $f_i$. In addition, if $M$ is a monomial generator of $I'$, then  
$$M= \frac{Na_1^{\deg_{b_1}N}\cdots a_\ell^{\deg_{b_{\ell}}N}}{b_1^{\deg_{b_1}N}\cdots b_{\ell}^{\deg_{b_{\ell}}N}},$$ 
for some monomial generator $N$ of $I$. If $x_0$ is the center of $f_i$ for some $i$, then $x_0 \mid M$ if and only if $x_0 \mid N$, in which case some other term of $f_i$ divides $N$ and thus $M$. Thus $f_i$ are still stars of degree one on $I'$. Moreover, given the form of $M$ and the fact that the stars $f_1,\ldots, f_q$ do not share any variables with the binomials $g_i$, it follows that the stars remain separated on $I'$. Now by \Cref{thm: stars reg seq}, $f_1, \ldots f_q$ form a regular sequence on $R'/I'$, and hence $g_1, \ldots, g_{\ell},f_1,\ldots,f_q$ is a regular sequence on $R/I$. 
\end{proof}

To illustrate the utility of \Cref{c:binom-extn}, we provide examples of linear binomial regular sequences. The first example will use {\it leaf pairs} for general, not necessarily square-free, monomial ideals, as defined in \cite{FHM}.

\begin{definition}[\protect{\cite[Definition~4.10]{FHM}}]
Let $I$ be a monomial ideal in a polynomial ring $R$. A variable $a\in R$ is called a {\it{ leaf }} in $I$ if there exists a unique monomial generator $M$ of $I$ such that $a\mid M$. Two variables $a,b$ are called a leaf pair in $I$ if $a,b$ are leaves in $I$ and there exist monomials $M_1, M_2, Z, W \in R$ with $\gcd(Z,W)=1$ such that $M_1=aZ$, $M_2=bW$, $\deg_{a}M_1=\deg_bM_2=1$, and $ZW\in\mathcal{G}(I)$.
\end{definition}

It was shown in \cite[Theorem~4.11]{FHM} that any sequence $\{a_i,b_i\}_{i=1}^{\ell}$ of disjoint leaf pairs forms a regular sequence $a_1+b_1, a_2+b_2, \ldots, a_{\ell} + b_{\ell}$ of degree one binomials. Thus \Cref{c:binom-extn} can be used to combine linear stars and leaf pairs to form regular sequences for general monomial ideals.

When $I$ is the edge ideal of a graph, there is a classification of linear binomials that are regular on $R/I$, which extends to sets of linear binomials that form a regular sequence, through the use of a property studied in graph theory referred to as ``Property $P$". Originally defined for a perfect matching, which then creates a maximal regular sequence, Property $P$ can be viewed as a property of individual edges and can be extended to graphs with loops, as was done in \cite{BrennanMorey1}, via the following definition.

\begin{definition}\cite[Definition 4.1]{BrennanMorey1}\label{Property P}
 Let $G$ be a graph, potentially with loops, and $e=\{x,y\}$ an edge of $G$. Then $e$ is said to have {\em Property $P$} if $\{x^2\}, \{y^2\} \not\in E(G)$ and for any vertices $a,b$ of $G$, if $\{a,x\}, \{y,b\}$ are edges of $G$, then $a \neq b$ and $\{a,b\}$ is an edge of $G$.
\end{definition}

By \cite[Theorem 3.6]{BrennanMorey1} for a simple graph and by \cite[Theorem 4.7]{BrennanMorey1} for a graph with loops, an edge $\{x,y\}$ has Property $P$ if and only if $x+y$ is regular on $R/I(G)$. More generally, by \cite[Theorem 6.1]{BrennanMorey1}, for any $x,y$ vertices of $G$, $x+y$ is regular if and only if either $\{x,y\}$ is an edge with Property $P$, or $\{x,y\}$ is not an edge, but in the graph $G'$ whose edges are $E(G) \cup \{x,y\}$, the new edge $\{x,y\}$ has Property $P$.
 Moreover, a set of disjoint edges that each satisfy Property $P$ leads to a regular sequence assuming that for each $j<i$ either $N_G(x_i) \cap \{x_j,y_j\} = \emptyset$ or $N_G(y_i) \cap \{x_j,y_j\}=\emptyset$  (see \cite[Theorem 5.7]{BrennanMorey1}). Thus for edge ideals of graphs, \Cref{c:binom-extn} can be used to combine linear stars and edges (or non-edges) with Property $P$ to form regular sequences.

%%%%%%%%%%%%%%%%%%%%%%%%%%%%%%%%%%

\section{Associated primes} \label{sec.AssPrimes}

In this section we investigate when a sequence of disjoint stars on an edge ideal of a graph $G$ forms a maximal regular sequence, thus realizing the depth of $R/I(G)$.  Recall that for a graph $G$, the \emph{star packing parameter}, $\alpha_2(G)$, is the maximum number of disjoint stars of degree one on $I(G)$. As a corollary of \Cref{thm: stars reg seq} one can see that when $I$ is the edge ideal of a graph $G$, then $\depth R/I \ge \alpha_2(G)$. This bound recovers earlier results of  \cite{DaoSchweig1} and \cite{FHM}. Moreover, it is also well-known that $\depth R/I(G)$ can be larger than $\alpha_{2}(G)$, as will be illustrated in \Cref{ex: letter H}. The first main result of this section will provide conditions under which this bound is sharp, meaning the depth of $R/I(G)$ is precisely $\alpha_2(G)$.

A key tool in proving the depth is realized by a particular regular sequence will be the observation that when a set of disjoint stars $f_1, \ldots, f_q$ is a maximal regular sequence, then $\depth R/(I, f_1, \ldots, f_q)=0$ and in particular, the homogeneous maximal ideal of $R$ becomes an associated prime of $(I, f_1,  \ldots, f_q)$. This naturally leads to the need to identify the associated primes of $(I, f_1, \ldots, f_q)$, which will be the second focus of this section.

\begin{example}\label{ex: letter H}
    Let $I(G)=(ab,bc,be,de,ef)$ be the edge ideal of the graph $G$ shown below.
    
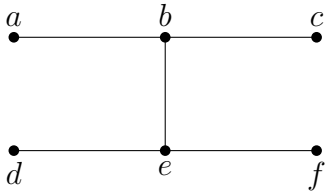
\begin{tikzpicture}[scale=1]

\fill (0,1.5) circle (2pt);
\fill (2,1.5) circle (2pt);
\fill (4,1.5) circle (2pt);
\fill (0,0) circle (2pt);
\fill (2,0) circle (2pt);
\fill (4,0) circle (2pt);

\node[above] at (0,1.5) {$a$};
\node[above] at (2,1.5) {$b$};
\node[above] at (4,1.5) {$c$};
\node[below] at (0,0) {$d$};
\node[below] at (2,0) {$e$};
\node[below] at (4,0) {$f$};

\draw (0,1.5) -- (2,1.5); 
\draw (2,1.5) -- (4,1.5); 
\draw (2,1.5) -- (2,0);   
\draw (0,0) -- (2,0);     
\draw (2,0) -- (4,0);     

\end{tikzpicture}

It is straightforward to verify that $\alpha_2(G)=2$. However, computation in Macaulay2 \cite{M2} shows that $\depth R/I=3$. The sequence $a+b,d+e, c+f$ is a maximal regular sequence, by \cite[Theorem~4.11]{FHM} and \Cref{c:binom-extn} since  $a+b, d+e$ are disjoint stars and $c, f$ are a pair of leaves of distance $3$ apart.
 \end{example}

As illustrated in \Cref{ex: letter H}, even though $\depth R/I\neq \alpha_2(G)$ it is still possible to produce a maximal regular sequence by combining disjoint stars with leaf pairs or non edges with Property $P$, see also \Cref{c:binom-extn}.

Our first goal is to establish conditions for the sequences of disjoint starts to form maximal sequences.  We first prove a straightforward observation for the edge ideal of a graph.

\begin{lemma}\label{l:max}
Let $I=I(G)$ be the edge ideal of a graph $G$ and let $f_1, \ldots, f_q$ be stars centered at variables $x_{1_0}, \ldots, x_{q_0}$. If $M_0$ is a monomial, then 
$$\bigcup_{\{i\;:\; x_{i_0}\mid M_0\}} (\supp(f_i)) \subseteq (I,f_1, \ldots, f_q):M_0.$$
\end{lemma}

\begin{proof}
First, note that if $M_0 \in (I,f_1, \ldots, f_q)$, then the result is trivial. So assume $M_0 \not\in (I,f_1, \ldots, f_q)$. Let $i\in [q]$ with $x_{i_0} \mid M_0$. Then for any $y \in \supp(f_i)$ if $y \ne x_{i_0}$, then $x_{i_0}y \in I$. Thus we have $y \in I : M_0 \subseteq (I,f_1, \ldots, f_q) : M_0$.  Since $f_i\in (I, f_1, \ldots, f_q):M_0$
 and for all $y \in \supp(f_i)$ if $y \ne x_{i_0}$, then  $y \in I : M_0 \subseteq (I,f_1, \ldots, f_q) : M_0$, then $x_{i_0} \in (I,f_1, \ldots,f_q):M_0$ as well.
\end{proof}

Note that if we have a collection of stars that use all the variables of the ring, then we obtain the following result. 

\begin{theorem}\label{prop: All variables covered}
Let $I=I(G)$ be the edge ideal of a graph $G$ in the polynomial ring $R$ over an algebraically closed field with homogeneous maximal ideal $\m$.  Let $f_1, \ldots, f_q$ be disjoint stars on $I$ such that $\m=\bigcup_{i=1}^q (\supp(f_i))$ and $M_0 = \prod_{i=1}^q x_{i_0} \not\in (I,f_1, \ldots, f_q)$. Then $\m \in \Ass(R/ (I, f_1, \ldots, f_q))$ and $\depth R/I=q$. 
\end{theorem}

\begin{proof}
Since $M_0 \not\in (I,f_1, \ldots, f_q)$, then $R \ne (I,f_1, \ldots, f_q) : M_0$. Since $\m=\bigcup_{i=1}^q (\supp(f_i))$, then $\m = (I,f_1, \ldots,f_q):M_0$, by \Cref{l:max}. 
Thus $\m \in \Ass(R/ (I, f_1, \ldots, f_q))$ 
and in particular, $\depth R/(I, f_1, \ldots, f_q)=0$. On the other hand,  since $f_1, \ldots, f_q$ are disjoint stars and $I$ is the edge ideal of a graph, then $f_1,\ldots,f_q$ are separated, so $f_1, \ldots, f_q$ is a regular sequence on $R/I$ by \Cref{thm: stars reg seq} and the conclusion follows. 
\end{proof}

Even though the 
hypothesis that $M_0 = \prod_{i=1}^q x_{i_0} \not\in (I,f_1, \ldots, f_q)$ is necessary in the proof of \Cref{prop: All variables covered}, it may still be possible for $\m$ to be an associated prime of $R/(I,f_1, \ldots, f_q)$ without this assumption, as seen in the following example. 

\begin{example}\label{ex: M in J}
   Consider the ideal
    $I=(ab,af,bc,bd, ce,cf,ch,ci,ed,df,hg,gi)$ in the polynomial ring $ R=k[a, \ldots, i]$ and let 
    $f_1=a+b+f$, $f_2=e+c+d$, and $f_3=g+h+i$. Then one can verify that  $M_0=aeg\in (I, f_1, f_2, f_3)$. 
    
    On the other hand, $f_1, f_2, f_3$ is a regular sequence on $R/I$ and one can check using Macaulay2 \cite{M2} that $\depth (R/I)=3$. Therefore, $\m \in \Ass(R/(I, f_1, f_2, f_3))$. Indeed, $\m = (I, f_1, f_2, f_3):ae$.
\end{example}

In both \Cref{prop: All variables covered} and \Cref{ex: M in J} the homogeneous maximal ideal is an embedded associated prime of $(I, f_1, \ldots, f_q)$. In the next result we show that we have a clear format for how the minimal primes of the ideal $(I, f_1, \ldots, f_q)$ arise. 

\begin{proposition}\label{lem: minimal prime format}
Let $I$ be a monomial ideal in a polynomial ring $R$ over a field $\KK$ and let $f_1, \ldots, f_q$ be homogeneous polynomials of degree one with pairwise disjoint support. Then
\begin{enumerate}[$($a$)$]
    \item For any monomial prime ideal $Q$, the ideal $(Q, f_1, \ldots, f_q)$ is a prime ideal.
    \item If $P\in \Min(R/(I, f_1, \ldots, f_q))$, then $P=(Q, f_1, \ldots, f_q)$, where  $Q\in \Min(R/I)$.
\end{enumerate}
\end{proposition}

\begin{proof}
\par (a) Let $Q$ be  a monomial prime ideal and consider the ideal $(Q, f_1, \ldots, f_q)$. For every $i\in [q]$ write  $f_i= \sum_j c_{i,j}x_{i,j}$, where $c_{i,j} \in \KK$ and $x_{i,j}$ are variables. Let 
$$\widehat{f}_i = \sum \limits_{x_{i,j}\in \supp(f_i)\setminus Q } c_{i,j}x_{i,j}.$$
Then $(Q, f_1, \ldots, f_q)=(Q, \widehat{f_1}, \ldots, \widehat{f_q})$, which is an ideal generated by the variables in $Q$ together with the linear sums $\widehat{f}_i$ where the sets $\supp(\widehat{f}_i)$, $\supp(\widehat{f}_j)$, $Q$ are pairwise disjoint for all $i \ne j$.   
Notice that $R/Q$ is a domain and the image of 
each $\widehat{f_i}$ in $R/(Q, \widehat{f_1}, \ldots, \widehat{f}_{i-1})$ is an irreducible polynomial. Therefore $R/(Q, \widehat{f_1}, \ldots, \widehat{f_q})$ is a domain and thus $(Q, f_1, \ldots, f_q)$ is a prime ideal.

\par (b) Let $P\in \Min(R/(I, f_1, \ldots, f_q))$. Since $I\subseteq (I, f_1, \ldots, f_q)$ and $P$ is prime there exists a minimal prime $Q$ of $I$ such that $Q\subseteq P$. Hence $(Q, f_1, \ldots, f_q)\subseteq P$. Now, by part (a) we know that $(Q, f_1, \ldots, f_q)$ is a prime ideal and $(I, f_1, \ldots, f_q) \subseteq (Q, f_1, \ldots, f_q)$. Thus $(Q, f_1, \ldots, f_q)=P$, by the minimality of $P$. 
\end{proof}

In the next example we show  that the converse of part  (b) in \Cref{lem: minimal prime format} does not hold in general. Computations for this example were done using Macaulay2 \cite{M2}.

\begin{example}\label{e:min to ass}
    We revisit \Cref{ex: letter H}. Notice that $a+b$ is a regular element on $R/I$. Note that $Q_1=(e,b)$ and $Q_2=(a,c,e)$ are minimal primes of $R/I$. It is straightforward to verify that $(Q_1, a+b)=(a,b,e)$ is a minimal  prime of $R/(I, a+b)$. On the other hand, $(Q_2,a+b)=(a,b,c,e)$ is an associated prime of $R/(I,a+b)$ and since $(Q_1, a+b)=(a,b,e)\in \Min(R/(I,a+b))$, then $(Q_2,a+b)=(a,b,c,e)$ is an embedded associated prime of $R/(I, a+b)$.
\end{example}

Note that \Cref{e:min to ass} gives an example of an ideal $I$ and a homogeneous polynomial $f$ of degree one for which there is a $Q \in \Min(R/I)$ with $(Q,f)\not\in \Min(R/(I,f))$. However, in this example, $(Q,f) \in \Ass(R/(I,f))$. Based on computational evidence, this occurs frequently, leading to the following conjecture.

\begin{conjecture} \label{conj.ass}
   Let $I=I(G)$ be the edge ideal of a graph $G$ in a polynomial ring $R$ over a field. Let $f_1,\ldots,f_q$ be separated stars on $I$. Then \begin{eqnarray*}
    P\in \Ass(R/(I,f_1,\ldots,f_q))  \text{ if and only if } P=(Q,f_1,\ldots,f_q),
\end{eqnarray*}
for some  $Q\in \Ass(R/I)=\Min(R/I)$.
\end{conjecture}

Although \Cref{conj.ass} is stated for edge ideals of graphs, it may also hold in greater generality for monomial ideals. However, the techniques developed in the remainder of this section are tailored specifically to edge ideals of graphs, which will be our focus here.

Supporting evidence for \Cref{conj.ass} is provided by the known description of associated primes of graphs with loops. When each of the $f_i$ is a binomial corresponding to a star on a graph $G$, that is, $f_i=x_i+y_i$ with $\{x_i,y_i\}\in E(G)$, and $y_i$ a leaf of $G$, then \Cref{conj.ass} can be seen to hold by applying \Cref{p:binomial version} to the situation where each $f_i$ is a star centered at a leaf $y_i$.  
Recall that for a graph $G$ and a vertex $x$, the neighborhood of $x$ in $G$ is defined by $N(x)=N_G(x) =\{a \in V(G) \mid \{a,x\} \in E(G)\}$ and the closed neighborhood of $x$ in $G$ is defined by $N[x]=N_G[x]=N(x) \cup \{x\}$.

\begin{proposition}\label{p:binomial version}
 Let $I = I(G)$ be the edge ideal of a graph $G$ in a polynomial ring $R$ over a field $\KK$. Let $f_1=x_1+y_1, \dots, f_q=x_q+y_q$ be disjoint binomial stars on $G$. Then, $P \in \Ass(R/(I,f_1, \dots, f_q))$ if and only if $P = (Q,f_1,\dots, f_q)$ for some $Q \in \Ass(R/I)= \Min(R/I)$.
\end{proposition}

\begin{proof} Without loss of generality, we assume that $y_i$ is a leaf for all $i \in [q].$
We proceed as in \Cref{c:binom-extn}, using the isomorphism shown in \cite{FHM}. More specifically, by \cite{FHM},  
$$R/(I,f_1,\ldots,f_q) \cong R'/I',$$
where $I'$ is the edge ideal of a graph $G'$ with loops formed from $G$ by replacing the edge corresponding to $x_iy_i \in I$ by a loop corresponding to $x_i^2\in I'$, and $R' \cong R/(y_1, \ldots, y_q)$. Note that given these isomorphisms, it follows that $P \in \Ass(R/(I,f_1,\ldots,f_q))$ if and only if $P= P_1 + (y_1, \ldots, y_q)$ for some $P_1 \in \Ass(R'/I')$.
By \cite[Lemma 4.2]{BrennanMorey1} (see also \cite[Corollary 4.14]{MV}), $P_1 \in \Ass(R'/I')$ if and only if $P_1=( W \cup (\cup_{i\in K} N_{G'}[x_i]))$ for some $K \subseteq [q]$ with $x_j, x_k$ not connected if $j,k \in K$ and $W$ is a minimal vertex cover of the edges of $G'$ not covered by $\cup_{i\in K} N_{G'}[x_i]$.

Assume $Q \in \Min(R/I)$ and let $P=(Q,f_1,\ldots, f_q)$. For each $i$, since $x_iy_i$ is a generator of $I$, either $x_i$ or $y_i$ is in $Q$. Since $f_i = x_i+y_i$, it follows that both $x_i, y_i \in P$ for all $i \in [q]$. Thus $P=P_1 + (y_1, \ldots, y_q)$, where $P_1 = Q \cup \{x_1, \ldots,x_q\} \setminus \{y_1,\dots, y_q\}$. Note that if $x_i \not\in Q$ for some $i \in [q]$, then $N_G(x_i) \in Q$ since $Q$ is a vertex cover of $G$. Thus $N_{G'}[x_i] \in P_1$ when $x_i \not\in Q$. If for some $i\ne j$, $x_i, x_j \not\in Q$ then since $Q$ is a vertex cover, $x_i$ and $x_j$ cannot be connected in $G$ or $G'$. Let $K \subseteq [q]$ be such that $$i\in K \text{ if and only if } x_i \not\in Q.$$ Set $W=P_1\setminus \cup_{i \in K} N_{G'}[x_i]$.
Then clearly, $P_1 = (W \cup (\cup_{i\in K} N_{G'}[x_i]))$. 

To complete the proof of this direction, we show that $W$ is a minimal vertex cover of the edges of $G'$ not covered by $\cup_{i\in K} N_{G'}[x_i]$. First we show that $W$ is a vertex cover of the edges not covered by $\cup_{i\in K} N_{G'}[x_i]$. Let $\{a,b\}$ be an edge of $G'$ such that $a,b \not\in (\cup_{i\in K} N_{G'}[x_i])$.  Then either $\{a,b\}$ is an edge of $G$, or $\{a,b\}$ is a loop of $G'$ and so $a=b=x_i$ for some $i\not\in K$. If $a=b=x_i$, then since $i\not\in K$, we have $a=x_i\in Q$. If $\{a,b\}$ is an edge of $G$, then
since $Q$ is a minimal vertex cover of $G$,  either $a$ or $b$ is in $Q$. By the definition of $G'$, neither $a$ nor $b$ is in $\{y_1, \ldots, y_q\}$. Hence, in either case, $a$ or $b$ is in $P_1$. Therefore, $a$ or $b$ is in $W$.

To see that $W$ is minimal, let $w \in W$. Since $W=P_1\setminus \cup_{i \in K} N_{G'}[x_i]$, we have $w \ne y_i$ for any $i$ and either $w \in Q$ or $w=x_i\not\in Q$ for some $i$. If $w=x_i\not\in Q$, then $i\in K$ and thus  $x_i\in (\cup_{i\in K} N_{G'}[x_i])$, a contradiction. Thus $w\in Q$. Since $Q$ is a minimal vertex cover, then there exists $z\in N_G(w)$ such that $z\not\in Q$. If $z=y_i$, then since $y_i$ is a leaf, $w=x_i$ and $w$ minimally covers the edge $x_i^2$ of $G'$. If $z \ne y_i$ for any $i\in [q]$, then $\{w,z\}$ is a non-loop edge of $G'$. Note that since $w\in W$, $z \ne x_i$ for $i\in K$, and more generally, $z \not\in (\cup_{i\in K} N_{G'}[x_i])$ since by definition, for $i\in K$, $N_G(x_i) \subseteq Q$ and $z \not\in Q$. Then $\{w,z\}$ is an edge of $G'$ not covered by $\cup_{i\in K} N_{G'}[x_i]$. We now claim that $z\not\in W$. Since $z \not\in Q$, if $z\in W$, then $z=x_j$ for some $j$. Then since $z=x_j \not\in Q$, we have $j\in K$ and $z\in \cup_{i\in K} N_{G'}[x_i]$, a contradiction. Thus $z \not\in W$ and $\{w,z\}$ is an edge of $G'$ minimally covered by $w$, as desired.

For the converse, assume $P \in \Ass(R/(I,f_1, \ldots, f_q))$. Then as before, $P=P_1 + (y_1,\ldots, y_q)$ where $P_1\in \Ass(R'/I')$. Note that since $f_i \in P$ and $y_i \in P$, then $x_i=f_i-y_i \in P$ for all $i\in [q]$. By \cite[Lemma 4.2]{BrennanMorey1}, $P_1=( W \cup (\cup_{i\in K} N_{G'}[x_i]))$ for some $K \subseteq [q]$ with $x_j, x_k$ not connected if $j,k \in K$ and $W$ a minimal vertex cover of the edges of $G'$ not covered by $\cup_{i\in K} N_{G'}[x_i]$. Define $Q_1=W \cup (\cup_{i\in K} N_G(x_i))$. Note that $N_{G}(x_i)=N_{G'}[x_i]\setminus\{x_i\}\cup \{y_i\}$. Then by definition, $Q_1$ is a vertex cover of $G$ and if $i\in K$, then $y_i \in Q_1$. Moreover, if $i \not\in K$, then $x_i \in Q_1$ since either $x_i \in N_G(x_j)$ for some $j \in K$, or $x_i \in W$ is needed to cover the loop $x_i^2$. Thus for each $i$, either $x_i$ or $y_i$ is in $Q_1$. To see that
$(Q_1, f_1,\ldots, f_q) \subseteq P$, first note that $x_i, y_i \in (Q_1, f_1,\ldots, f_q)$ for all $i \in [q]$, and $(Q_1, x_1,\ldots, x_q, y_1,\ldots, y_q)$ is a (non-minimal) generating set. Let $b \in(Q_1, f_1,\ldots, f_q)$. If $b=y_i$ or $b=x_i$, then $b\in P$. If $b \in Q_1$ but $b\ne y_i$, then $b\in P_1 \subseteq P$. For the reverse inclusion, let $a \in P$. If $a=x_i$ or $a=y_i$, then $a\in (Q_1, f_1,\ldots, f_q)$. If $a \in P_1$ but $a \ne x_i$, then $a \in ( W \cup (\cup_{i\in K} N_{G'}(x_i))) \subseteq Q_1$.

Finally, we claim that $Q_1 \in \Min(R/I)$. As above, $Q_1$ is a vertex cover. Let $w \in Q_1$. If $w \in W$, then by the definition of $W$, there is an edge $wz$ with $z \not\in W$ and $z \not\in (\cup_{i\in K} N_{G'}[x_i]))$. Thus $z \not\in Q_1$. If $w \not\in W$, then $w \in (\cup_{i\in K} N_{G'}(x_i))$ and the edge $wx_i$ for some $i\in K$ is minimally covered by $w$ since $x_i \not\in Q_1$.
\end{proof}

When allowing for more general stars that are not necessarily binomial, some additional hypotheses are needed to prove that ideals of the form $(Q,f_1,\ldots, f_q)$ with $Q\in \Min(R/I)$ are associated to $(I,f_1, \ldots, f_q)$. These conditions are introduced below, and will be used in \Cref{thm: minimal primes extend} to provide additional supporting evidence for \Cref{conj.ass}.

\begin{setting}\label{set u_i}
Let $I=I(G)$ be the edge ideal of a graph $G$ in a polynomial ring $R$ over a field $\KK$ and let $f_1, \ldots, f_q$  be pairwise disjoint stars on $I$. Fix $Q\in \Min(R/I)$ and let
$\bar f_i$ denote the image of $f_i$ in $R/Q$.
Set $D=V(G)\setminus Q$. 

After renumbering, assume that there exists $r\ge 0$ such that 
$\bar f_i$ is a variable for $i\in[r]$,
and that $\bar f_i$ is a nonzero non-variable linear form for $r < i \le q$.

Further assume that for each $i\in[r]$ there exists a vertex $u_i$ such that 
$ u_i \in \text{Supp}(f_i)\cap Q
$ with $N_G(u_i)\cap D = \{\bar f_i\}$
and that the vertices $u_1,\dots,u_r$ are pairwise nonadjacent.
\end{setting}

We are now able to show that one direction of \Cref{conj.ass} is true when the assumptions of \Cref{set u_i} hold. 

\begin{theorem}\label{thm: minimal primes extend}
Under the hypotheses of \Cref{set u_i}, 
\[
(Q,f_1,\ldots,f_q)\in \Ass(R/(I,f_1,\ldots,f_q)).
\]
\end{theorem}

\begin{proof}
By abuse of notation, let $R=\KK[x_1,\ldots,x_n]$, where $V(G)=\{x_1,\ldots,x_n\}$ is the set of vertices of $G$. Let $Q\in \Min(R/I)$ and recall that $Q$ is a minimal vertex cover of $G$. Set $J=(I,f_1,\ldots,f_q)$, $P=(Q,f_1,\ldots,f_q)$, and $D=V(G)\setminus Q$. By \Cref{lem: minimal prime format}, $P$ is a prime ideal.

For each $i\in [q]$, let $\bar f_i$ denote the image of $f_i$ in $R/Q$. As in \Cref{set u_i}, after renumbering, there exists $r\ge 0$ such that $\bar f_i=t_i$ is a variable for $1\le i\le r$, and $\bar f_i$ is a nonzero non-variable linear form for $r<i\le q$. Notice that $P=(Q,f_1,\ldots,f_q)=(Q,\bar f_1,\ldots,\bar f_q)$.

Set $S=(R/J)_P$. Since $P\supseteq J$ and $P$ is prime, we have $S\neq 0$.

We first handle the case $r=0$. In this case, every $\bar f_i$ is a nonzero non-variable linear form. Since the stars have pairwise disjoint supports, no variable $x\in D$ lies in the ideal $(\bar f_1,\ldots,\bar f_q)$ of $R/Q$. Hence $x\notin P$ for every $x\in D$, and so each such $x$ is a unit in $S$. Now let $y$ be a variable in $Q$. Since $Q$ is a minimal vertex cover of $G$, there exists $x\in D$ such that $xy\in I$. Thus $xy=0$ in $S$, and since $x$ is a unit in $S$, we get $y=0$ in $S$. Therefore $QS=0$. Since each $f_i=0$ in $S$, we have $PS=0$. As $S\neq 0$, it follows that $PS\in \Ass_S(S)$, and hence $P\in \Ass_R(R/J)$.

Now suppose $r\ge 1$ and let $T=\{t_1,\ldots,t_r\}$. As above, every variable $x\in D\setminus T$ is a unit in $S$. Set $M=\prod_{t_i\in T} t_i$. We first show that $PS\subseteq \Ann_S(M)$. Let $x\in Q$. By the minimality of $Q$, there exists a neighbor $z$ of $x$ with $z\in D$. If $z\in D\setminus T$, then $z$ is a unit in $S$, and since $xz\in I\subseteq J$, we get $x=0$ in $S$. Suppose instead that all neighbors of $x$ in $D$ lie in $T$. Then $x$ has a neighbor $t_i\in T$, so $xt_i\in I$ and hence $xt_i=0$ in $S$. Thus $xM=0$. Therefore $QS\subseteq \Ann_S(M)$. Since $f_i=0$ in $S$ for all $i$, we have $PS\subseteq \Ann_S(M)$. It remains to prove that $M\neq 0$ in $S$.

Let $\FF=\operatorname{Frac}(R/P)$ and let $B=\FF[\varepsilon_1,\ldots,\varepsilon_r]/(\varepsilon_1^2,\ldots,\varepsilon_r^2)$. Define a homomorphism $\psi:R\to B$ by setting $\psi(t_i)=\varepsilon_i$, $\psi(u_i)=-\varepsilon_i$, $\psi(x)=0$ for $x\in Q\setminus\{u_1,\ldots,u_r\}$, and $\psi(y)=\bar y\in \FF$ for $y\in D\setminus T$.

We claim that $\psi(f_i)=0$ for all $i$. For $1\le i\le r$, we may write
\[
f_i=t_i+u_i+\sum_{x\in \supp(f_i)\cap Q\setminus\{u_i\}}x,
\]
and hence $\psi(f_i)=\varepsilon_i-\varepsilon_i=0$. For $r<i\le q$, the support of $f_i$ is disjoint from the supports of $f_1,\ldots,f_r$, so no variable $t_j$ or $u_j$ with $j\le r$ occurs in $f_i$. Thus $\psi(f_i)$ is the image of $f_i$ in $\FF=\operatorname{Frac}(R/P)$, which is zero because $f_i\in P$.

Next we claim that $\psi(I)=0$. Let $ab\in I$ be an edge of $G$. Since $Q$ is a vertex cover, we may assume that $a\in Q$. If $a\neq u_i$ for all $i$, then $\psi(a)=0$, and so $\psi(ab)=0$. Suppose $a=u_i$ for some $i\le r$. If $b\in D$, then the hypothesis $N_G(u_i)\cap D=\{t_i\}$ gives $b=t_i$, and hence $\psi(ab)=-\varepsilon_i^2=0$. If $b\in Q$, then the pairwise nonadjacency of $u_1,\ldots,u_r$ implies that $b\neq u_j$ for all $j$. Hence $\psi(b)=0$, and again $\psi(ab)=0$. Thus $\psi(I)=0$, and so $\psi$ induces a homomorphism $\overline{\psi}:R/J\to B$.

We now show that $\overline{\psi}$ localizes at $P$. Let $\pi:B\to B/(\varepsilon_1,\ldots,\varepsilon_r)\cong \FF$ be the quotient map. Observe that if $g\notin P$, then $\pi(\psi(g))\neq 0$. Since $\FF$ is a field and $(\varepsilon_1,\ldots,\varepsilon_r)$ is a nilpotent ideal of $B$, it follows that $\psi(g)$ is a unit in $B$. Therefore $\overline{\psi}$ induces a homomorphism $\widetilde{\psi}:S=(R/J)_P\to B$.

Finally, $\widetilde{\psi}(M)=\psi(t_1\cdots t_r)=\varepsilon_1\cdots\varepsilon_r$. This element is nonzero in $B$, since the only relations among the $\varepsilon_i$ are $\varepsilon_i^2=0$. Thus $M\neq 0$ in $S$. Since $PS\subseteq\Ann_S(M)$ and $M\neq 0$, the annihilator $\Ann_S(M)$ is a proper ideal containing the maximal ideal $PS$ of the local ring $S$. Hence $\Ann_S(M)=PS$. Therefore $PS\in\Ass_S(S)$, and it follows that $P\in\Ass_R(R/J)$.
\end{proof}

%%%%%%%%%%%%%%%%%%%%%%%%%%%%%%%%

\bibliography{RegularBib}
\bibliographystyle{plain}

\end{document}